\documentclass[11pt,a4paper]{amsart}
\usepackage{iftex}
\ifluatex
    \usepackage{luatexja}
\fi
\ifpdftex
\fi

\usepackage{amsmath,amssymb,amsthm}
\usepackage[alphabetic,nobysame]{amsrefs}
\usepackage{amsaddr}
\usepackage{tikz,tikz-cd}

\usepackage{mathtools}
\usepackage{hyperref}
\usepackage{myarticlepreamble}
\theoremstyleenglish

\begin{document}

\title{On torsion-free modules and semi-hereditary rings}
\author{Ryoya ANDO}

\address{\tusaddress}
\address{SKILLUP NeXt Ltd., Chiyoda, Japan}
\email{\url{6122701@ed.tus.ac.jp}}
\email{\url{r_ando@skillupai.com}}
\date{\today}


\def\labelenumi{(\theenumi)}
\begin{abstract}
The class of semi-hereditary rings is an important class of rings in theories that do not assume the Noetherian condition, such as perfectoid ring
theory. We prove several results concerning the structure theory of this class, focusing on the relationship between semi-hereditary rings and the flatness of torsion-free modules. We also consider Shimomoto's problem concerning the flatness of the Frobenius map.

\end{abstract}
\maketitle
\section{Introduction}
	
Throughout this paper, all rings are commutative with the identity element. 

In recent years, rings without Noetherian conditions have been increasingly used in homological studies in commutative ring theory, exemplified by perfectoid ring theory. The class of coherent rings, which also appears in the context of perfectoid ring theory, is one of the important classes of rings that includes the class of Noetherian rings. For example, it is important in applications such as \cite{AMM} and \cite{SHIMOMOTO201124}. 

We discuss several observations on semi-hereditary rings, a subclass of coherent rings, which can be regarded as a generalization of Dedekind domains.

In \S \ref{sec:semi-here and torsion-free} and \S\ref{sec:Integral domaincases}, we give a brief review of semi-hereditary rings, torsion-free modules, and some properties of Pr\"ufer domains.

In \S \ref{sec:general cases},  we prove that the ring $A$ is semi-hereditary if and only if every torsion-free $A$-module is flat (\ref{thm:main}). This theorem is an extension of Chase's theorem (\cite{Chase}*{Theorem 4.1}) 
, \ref{thm:semi-here <=> torless is flat} of this paper).

In \S \ref{sec:decomp}, we will consider propositions concerning the decomposition of semi-hereditary rings into direct products in certain cases. While there are previous studies on the decomposition theorem for semi-hereditary rings, such as \cite{Faith1975}, \cite{Drozd1980},  and \cite{MO2001}, we will attempt a different approach from those.

Finally, in \S \ref{sec:perfectoid}, we show that the main theorem can be applied to several propositions in perfectoid ring theory. In perfectoid ring theory, torsion-free modules play an important role in several propositions originating from \cite{Lurie}*{Theorem 3.2.}, inspired by O. Gabber. The results of this work solve \cite{shimomoto}*{Open Problem 7.13.}.

\section{Semi-hereditary rings and torsion-free modules}\label{sec:semi-here and torsion-free}

As for coherent rings, a well-known textbook is \cite{Glaz}, which summarizes the properties of semi-hereditary rings. Additionally, the classes of rings such as semi-hereditary, Pr\"ufer, and von Neumann regular that appear in this paper are also well summarized in \cite{Lam1999}. We will recall the definitions of these concepts.

\begin{defi}
    Let $A$ be a ring. $A$ is called \textbf{hereditary} if every ideal of $A$ is projective. Also, $A$ is called \textbf{semi-hereditary} if every finitely generated ideal of $A$ is projective.
\end{defi} 

As basic examples, valuation rings and Dedekind domains are semi-hereditary.

We will introduce the characterization of semi-hereditary rings given by \cite{Chase}. Recall that an $A$-module $M$ is called \textbf{torsion-less} if the following map is injective;
\[M\to \hom(\hom(M,A),A); x\mapsto (f\mapsto f(x)).\] 

\begin{thm}[\cite{Chase}*{theorem 4.1}]\label{thm:semi-here <=> torless is flat}
	A ring $A$ is semi-hereditary if and only if every torsion-less $A$-module is flat.
\end{thm}

Let us also consider torsion-free modules here. Let us recall the definition.

\begin{defi}
    An $A$-module $M$ is called \textbf{torsion-free} if for any non-zero divisor $a \in A$, the natural map $a\cdot: M \to M$ is injective. 
\end{defi}

Obviously, a torsion-less module is torsion-free, and a flat module is also torsion-free. When does the converse hold? 

For example, the $\Z$-module $\Q$ is torsion-free but not torsionless (since $\hom_\Z(\Q, \Z) = 0$). So, when does a torsion-free module become flat? Chase's theorem (\cite{Chase}*{Theorem 4.2}, \ref{thm:Chase} of this paper) characterizes rings where flatness and torsion-free-ness are equivalent in the case of integral domains. Such rings are Pr\"ufer domains.

Here, by \ref{thm:semi-here <=> torless is flat}, we pay attention to the fact that a ring where flatness and torsion-free-ness are equivalent is semi-hereditary. In this paper, we extend this theorem and prove that the class of rings where flatness and torsion-free-ness are equivalent, even in the general case, is exactly the class of semi-hereditary rings (\ref{thm:main}). We also summarize the situation for the integral domain case, Noetherian case, and general case.

\section{Integral domain cases}\label{sec:Integral domaincases}

In the case of integral domains, the key concept is that of Pr\"ufer domains, which are known as a generalization of Dedekind domains. Let us begin by reviewing the definition.

\begin{defi}
	A domain $A$ is called a Pr\"ufer domain, if for any prime ideal $P$ of $A$, $A_P$ is a valuation ring.
\end{defi}

This condition is equivalent to any finitely generated ideal being projective. Although this fact is well known, We will provide a proof.

\begin{prop}\label{prop:cond. of Prufer}
	Let $A$ be a integral domain. The following conditions are equivalent:
	\begin{sakura}
		\item $A$ is a  Pr\"ufer domain;
		\item Every finitely generated ideal is projective;
		\item Every ideal is flat.
	\end{sakura}	
\end{prop}

\begin{proof}
	\begin{eqv}[3]
 
        \item Let $I$ be a finitely generated ideal of $A$. For any $P \in \spec A$, $I_P$ is a finitely generated ideal of $A_P$, so it is principal. Since $A_P$ is a domain, $I_P$ is free, and hence projective. Since projectivity is a local property, $I$ is projective. 
        \item Let $I$ be an ideal of $A$. It can be expressed as $I = \ilim I_i$ by a family of finitely generated ideals $\{I_i\}$. Since each $I_i$ is projective, for any $A$-module $M$, we have $\Tor_1(I, M) = \Tor_1(\ilim I_i, M) = \ilim \Tor_1(I_i, M) = 0$. Hence, $I$ is flat. 
        \item For any $P \in \spec A$ and finitely generated ideal $J$ of $A_P$, $J$ is the localization of an ideal of $A$, so it is flat. Therefore, by Kaplansky's theorem \cite{Kaplansky1958}, it must be free and principal. Hence, $A_P$ is a valuation ring. 
	\end{eqv}
\end{proof}

Let us give a simple proof of \cite{Chase}*{Theorem 4.2}.

\begin{thm}[\cite{Chase}*{Theorem 4.2}] \label{thm:Chase}
	Let $A$ be a domain. The following conditions are equivalent:
	\begin{sakura}
		\item $A$ is a Pr\"ufer domain;
		\item Every torsion-free $A$-module is flat.
	\end{sakura}  
\end{thm}

\begin{proof}
    (i)$\Longrightarrow$ (ii): Let $M$ be a torsion-free $A$-module. It suffices to show for any $P\in\spec A, M_P$ is flat. Now, $M_P$ is a torsion-free $A_P$ module. This is because if $a/s \neq 0 \in A_P$ and $x \in M$, and if $ax/s = 0$, then for some $h \not\in P$, we have $ahx = 0$. Since $A$ is a domain and $a/s \neq 0$ implies $ah \neq 0$, and because $M$ is torsion-free, it follows that $x = 0$. Now, if $J$ is a finitely generated ideal of $A_P$, since $A_P$ is valuative, $J$ is a principal ideal. Therefore, since $M_P$ is torsion-free, the map $J \otimes M_P \to M_P$ is injective. This shows that $M_P$ is flat.
	
    (ii)$\Longrightarrow$(i): It suffices to show for any $P\in\spec A, A_P$ is a valuative ring. Let $J$ be a finitely generated $A_P$ ideal. Let $I$ be the pullback of $J$ along the natural map $A \to A_P$. Since $A$ is a domain, $I$ is torsion-free, i.e., flat. Hence, $J = I_P$ is also flat, and since $A_P$ is local, $J$ is free. Therefore, $J$ is principal, and $A_P$ is valuative.
\end{proof}

The key point of this proof is that in a domain, the localization of a torsion-free module remains torsion-free. This perspective is important when extending this theorem. See \ref{prop: A reduced and Q(A) v.N.r => tf is local condition}.

Let us organize the relationships between Pr\"ufer domains and other classes. The following two facts are well known.

\[\text{Noetherian Pr\"ufer domain $\Longleftrightarrow$ Dedekind domain.}\]
\[\text{Pr\"ufer domain $\Longleftrightarrow$ semi-hereditary domain.}\]

\section{General Cases}\label{sec:general cases}

We will now discuss the case where we do not assume the integral domain condition, considering both the case where Noetherian condition is imposed and the more general case where it is not.

\begin{prop}\label{prop:(tf->f) => reduced}
    Let $A$ be a ring. If every torsion-free $A$-module is flat, then $A$ is reduced.
\end{prop}

\begin{proof}
    Regarding $\nil(A) = \mkset{a \in A}{a^n = 0~\text{for some } n > 0}$, $A/\nil(A)$ is a torsion-free $A$-module. By assumption, $A/\nil(A)$ is flat, so for any $a \in \nil(A)$, the natural map $aA \otimes_A (A/\nil(A)) \to A/\nil(A)$ is injective. However, this is a zero map, so $aA \otimes A/\nil(A) = 0$. Therefore, we have $aA = \nil(A) \cdot aA$, and by NAK \cite{Matsumura1986}*{Theorem 2.2}, it follows that $aA = 0$. Hence, $\nil(A) = 0$.
\end{proof}

This proposition shows that in the case where Krull dimension is 0, the class of rings for which torsion-free modules are flat is none other than the class of von Neumann regular rings. Let us recall the definition. 

\begin{defi}
    A ring $A$ is called \textbf{von Neumann regular} (or absolutely flat) if for any element $a \in A$, there is a $b \in A$ satisfying $a^2b = a$.
\end{defi}

\begin{cor}\label{cor:v.N.r <=> (tf->f) under dim = 0}
	Let $A$ be a ring with $\dim A= 0$. The following conditions are equivalent:
	\begin{sakura}
		\item $A$ is von Neumann regular;
		\item Every torsion-free $A$-module is flat.
	\end{sakura}
\end{cor}
	
\begin{proof}
(i) $\Longrightarrow$ (ii)

If $A$ is von Neumann regular, then by \cite{Glaz}*{Theorem 1.4.6}, any $A$-module is flat.

(ii) $\Longrightarrow$ (i)

By \ref{prop:(tf->f) => reduced}, $A$ is reduced. A ring $A$ being reduced and having $\dim A = 0$ is equivalent to being von Neumann regular \cite{Lam1999}*{Theorem 3.71}, so $A$ is von Neumann regular.
\end{proof}

We introduce a restatement of the property of being semi-hereditary, as described by Glaz.

\begin{thm}[\cite{Glaz}*{Corollary 4.2.19}]
	Let $A$ be a ring. The following conditions are equivalent:
	\begin{sakura}
		\item $A$ is semi-hereditary.
		\item The total quotient ring of $A$, $Q(A)$ is von Neumann regular and for any maximal ideal $\ideal{m}$ of $A$, $A_{\ideal{m}}$ is a valuation ring.
	\end{sakura}
\end{thm}

We note that if $A$ is a reduced ring, then $Q(A)$ is also reduced. Therefore, if $\dim Q(A) = 0$, $Q(A)$ is von Neumann regular. Furthermore, when $A$ is a Noetherian reduced ring, we have $\dim Q(A) = 0$ (see \cite{Matsumura1986}*{Exercise 6.5}). But note that this is not always the case in general, see \ref{ex:reduced and dim Q(A) =2} and \cite{Quentel1971}.

The condition that the dimension of total quotient ring  is $0$ plays an important role in ensuring that the localization of a torsion-free module remains torsion-free.

\begin{prop}\label{prop: A reduced and Q(A) v.N.r => tf is local condition}
    Let $A$ be a ring with $\dim Q(A) = 0$. For any torsion-free $A$-module $M$ and $P\in\spec A$, $M_P$ is a torsion-free $A_P$-module.
\end{prop}

\begin{proof}
    Let $\alpha = a/s$ be a non-zero divisor in $A_P$, and $\bar{x} = x/t \neq 0 \in M_P$ with $\alpha \bar{x} = 0$. Then, there exists some $h_1 \not\in P$ such that $h_1 a x = 0$. 
    
    We define $I \coloneq \ker(A \to A_P) = \mkset{a \in A}{ha = 0~\text{for some}~h \not\in P}$.

    Assume that $(a, I) Q(A) \neq Q(A)$. Then, there exists a maximal ideal $\ideal{m}$ of $Q(A)$ containing $(a, I)$. Since $\dim Q(A) = 0$, $\ideal{m}$ is minimal, so if we let $P' \coloneq \ideal{m} \cap A \in \spec A$, then $P'$ is a minimal prime ideal containing $(a, I)$. Here, $P' A_{P'}$ is the unique prime ideal of $A_{P'}$, so all elements are nilpotent. Consequently, there exists some $n > 0$ and $b \not\in P'$ such that $ba^n = 0$. Now, since $a/s$ is a non-zero divisor in $A_P$, it follows that $ba^{n-1}/1 = 0 \in A_P$. Repeating this, we get $b/1 = 0 \in A_P$. In other words, $b \in I \subset P'$, which is a contradiction.
    
    Therefore, $(a, I) Q(A) = Q(A)$. This implies that there exists some $b \in (a, I)$ which is a non-zero divisor of $A$. In this case, let $b = ra + c$ for some $r \in A$ and $c \in I$. Since there exists some $h_2 \not\in P$ such that $h_2 c = 0$, for $h \coloneq h_1 h_2 \not\in P$, we have $h b x = 0$. Now, since $b$ is a non-zero divisor of $A$ and $M$ is a torsion-free $A$-module, it follows that $h x = 0$. Thus, $\bar{x} = 0 \in M_P$, and $M_P$ is a torsion-free $A_P$-module.
\end{proof}

In \ref{prop: A reduced and Q(A) v.N.r => tf is local condition}, it remains a question whether the condition on the dimension of the total quotient ring can be weakened. 

\begin{prob}
    Characterize a ring $A$ with the following property: If $M$ is any
torsion free $A$-module and $P\in\spec A$, then $M_P$ is torsion-free $A_P$-module.
\end{prob}

Several examples are presented where the localization of a torsion-free module is not torsion-free. The following example is given in a discussion we had with Ohashi and Matsumoto.

\begin{example}[Noetherian and not reduced Example]\label{ex:Noeth and reduced}
    Let $k$ be a field. Let $\ideal{m} = (x,y,z) \subset k[x,y,z]$, and define $A = k[x,y,z]_{\ideal{m}} /(xy,yz,z^2)$. For $P = (x,y)\subset A$ and $M = A/xA$, the module $M$ is a torsion-free $A$-module, but $M_P$ is not a torsion-free $A_P$-module. Now, a regular element of $A$ is a unit, so $Q(A) = A$ and $\dim Q(A) = 2$.
\end{example}

\ref{ex:Noeth and reduced} is an example in a Noetherian ring. As mentioned earlier, if a ring $A$ is Noetherian, then if it is reduced, then $\dim Q(A) = 0$. Thus, to serve as a counterexample in a Noetherian ring, it must not be reduced. An example in a reduced ring is the following.

\begin{example}[non Noetherian and reduced Example]\label{ex:reduced and dim Q(A) =2}
    Let $k$ be a field. Define the ideals $I$ of the infinite-variable polynomial ring $k\left[X_n\mathrel{}\middle|\mathrel{}n = 1,2,3,\dots\right]$ as:
    \[ I = \left(X_iX_j \mathrel{}\middle|\mathrel{}\text{$i\neq j$ and $\{i,j\}\neq\{1,2\}$}\right).\]
    Minimal ideals of $A$ are as follows:
    \[P_{1,2}\coloneq \left(X_n \mathrel{}\middle|\mathrel{} n\neq 1,2\right),P_i\coloneq \left(X_n \mathrel{}\middle|\mathrel{} n\neq i\right) \text{ for each } i\geq 3.\]
    A regular element of $A$ is a polynomial with non-zero constant term. For $P \coloneq P_{1,2} +(X_2) = \left(X_n \mathrel{}\middle|\mathrel{} n\neq 1\right)$ and $M =  A/X_2A$, the module $M$ is a torsion-free $A$-module, but $M_P$ is not a torsion-free $A_P$-module. This ring $A$ is reduced unlike \ref{ex:Noeth and reduced}, and $\dim Q(A) = 2$.
\end{example}

Note that the converse of \ref{prop: A reduced and Q(A) v.N.r => tf is local condition} does not hold.

\begin{example}
    Let $k$ be a field. Let $\ideal{m} = (x,y) \subset k[x,y]$, and define $A = k[x,y]_{\ideal{m}} /(x^2,xy)$.
    The prime ideals of $A$ are $P = (x)$ and $\ideal{m} = (x,y)$ only. Since $A_\ideal{m} = A$ and $A_P = k(y)$, the regular elements in these rings are units. Therefore, the localization of any torsion-free module is also torsion-free.

    On the other hand, as $A = Q(A)$ and $\dim Q(A) = 1$, this provides a counterexample to the converse of \ref{prop: A reduced and Q(A) v.N.r => tf is local condition}.
\end{example}

We use \ref{prop: A reduced and Q(A) v.N.r => tf is local condition} to characterize rings in which torsion-free modules become flat.

\begin{thm}\label{thm:main}
    Let $A$ be a ring . The following conditions are equivalent:
    \begin{sakura}
		\item $A$ is semi-hereditary ring.
		\item Every torsion-free $A$-module is flat.
    \end{sakura}
\end{thm}

\begin{proof}
    \begin{eqv}
        \item Let $M$ be a torsion-free $A$-module. Since a semi-hereditary ring is reduced and $\dim Q(A) = 0$, by \ref{prop: A reduced and Q(A) v.N.r => tf is local condition}, for any maximal ideal $\ideal{m}$, $M_{\ideal{m}}$ is a torsion-free $A_{\ideal{m}}$-module. Since $A_{\ideal{m}}$ is a valuation ring, by an argument similar to the proof of \ref{thm:Chase}, $M_{\ideal{m}}$ is flat. Therefore, $M$ is also flat.

        \item If $M$ is a torsion-less module, then since $M$ is torsion-free, it follows that $M$ is flat. Therefore, by \ref{thm:semi-here <=> torless is flat}, $A$ is semi-hereditary.
    \end{eqv}
\end{proof}

\section{decomposition theorem for semi-hereditary rings}\label{sec:decomp}

From the above, we have been able to reformulate a semi-hereditary ring in terms of torsion-free modules.  

Using an idea similar to \ref{prop:(tf->f) => reduced}, we can prove that a ring decomposes if there exists a special prime ideal. Since flatness is a local property, note that the following lemma holds.

\begin{lem}
	Let $A$ be a semi-hereditary ring, i.e. a ring with every torsion-free $A$-module is flat. If there exist rings $A_1$ and $A_2$ such that $A \cong A_1 \times A_2$, then they are semi-hereditary.
\end{lem}

\begin{prop}
	Let $A$ be a semi-hereditary ring. Suppose that there exists a prime ideal $P\in\spec A$ such that every element of $P$ is a zero divisor. Then, $P^2 = P$. Furthermore, if $P$ is finitely generated, the following holds. 
	\begin{sakura} 
		\item There exists an idempotent element $e\in P$ such that $P = (e)$.
		\item $A$ decomposes as $A\cong A/eA \times A/(1-e)A$. 
		\item $A/eA$ is a Pr\"ufer domain, and $A/(1-e)A$ is semi-hereditary.

	\end{sakura}
\end{prop}

\begin{proof}
	Since all elements of $P$ are zero divisors, $A/P$ is a torsion-free $A$-module. Thus, $A/P$ is flat, just like \ref{prop:(tf->f) => reduced},  the natural map $P \otimes_A A/P \to A/P$ is injective and zero. Therefore, $P \otimes A/P = P/P^2 = 0$.

	Now, suppose $P$ is finitely generated.
	\begin{sakura}
		\item By NAK \cite{Matsumura1986}*{Theorem 2.2}, there exists some $a \in P$ such that $(a + 1)P = 0$. Thus, $P = aA$. Since $P = P^2$, there exists some $b \in A$ such that $a = a^2b$. Letting $e = ab$, we see that $e$ is idempotent and $P = aA = eA$.
		
		\item Let $A\to A/eA \times A/(1-e)A; a\mapsto (a+eA, a+(1-e)A)$ and $A/eA \times A/(1-e)A\to A;(b+eA,c+(1-e)A)\mapsto b(1-e)+ce$ be homomorphisms that are inverses of each other.

        \item From the lemma, $A/eA$ and $A/(1-e)A$ are semi-hereditary. Here, since $A/eA = A/P$ is a domain, $A/eA$ is a Prüfer domain by \ref{thm:Chase}.
	\end{sakura}
\end{proof}

For a ring $A$, it is important to note that there always exists a minimal prime ideal, and all elements of the minimal prime ideal are zero divisors. Moreover, according to \cite{Dutton1978}*{Proposition 4,5}, if the minimal prime ideal $P$ is finitely generated, then $P$ becomes an associated prime, meaning that there exists some $a \neq 0 \in A$ such that $P = \ann(a)$.

\begin{cor}\label{cor:min primes is finite => A is a fin. prod. of Prufer}
	Let $A$ be a semi-hereditary ring. If every minimal prime ideal of $A$ is finitely generated, then $A$ is a finite product of Pr\"ufer domains.
\end{cor}

\begin{proof}
	Under this assumption, by \cite{Anderson1994}, the minimal prime ideals are finite in number. Since each decomposition of the ring $A$ as $A \cong A/P \otimes A/(1-e)A$ using the minimal prime $P$ decreases the number of minimal primes, this decomposition can occur at most finitely many times.
\end{proof}

From the above we can conclude that in the case of Noetherian rings, semi-hereditary rings are expressed as finite products of Dedekind domains. Before giving the proof, we will briefly summarise what is known about the relations between different classes.

\begin{lem}
	Let $A$ be a ring. The following conditions are equivalent:
	\begin{sakura}
		\item $A$ is a Noetherian Pr\"ufer domain;
		\item $A$ is a hereditary domain;
		\item $A$ is a Dedekind domain.
	\end{sakura}
\end{lem}

\begin{proof}
	\begin{eqv}[3]
		\item Any ideal of $A$ is finitely presented and flat, so it is projective. Therefore, $A$ is hereditary.
		\item Any ideal of $A$ is projective.  Since $A$ is a domain, (fractional) ideal of $A$ is projective if and only if invertible (see \cite{Matsumura1986}*{Theorem 11.3}). So $A$ is a Dedekind domain.
        \item Since a Dedekind domain is Noetherian and every ideal of $A$ is projective, so $A$ is a Pr\"ufer domain.
	\end{eqv}
\end{proof}

\begin{cor}
	Let $A$ be a Noetherian ring. The following conditions are equivalent:
	\begin{sakura}
		\item $A$ is a semi-hereditary ring.
		\item $A$ is a finite direct product of Dedekind domains.
		\item $A$ is a hereditary ring.
	\end{sakura}
\end{cor}

\begin{proof}
	\begin{eqv}[3]
	   \item By \ref{cor:min primes is finite => A is a fin. prod. of Prufer}, $A$ is a finite product of Prüfer domains. Now, from the construction, each Prüfer domain component of $A$ is Noetherian. In a Noetherian Prüfer domain, every ideal is finitely presented and flat, hence projective. Therefore, $A$ is a finite product of hereditary domains. Since hereditary domains are equivalent to Dedekind domains, $A$ is a finite product of Dedekind domains.	
        \item A direct product of hereditary rings is also hereditary.
        \item It is obvious.
	\end{eqv}
\end{proof}

\section{Some topics in perfectoid ring theory}\label{sec:perfectoid}

It goes without saying that the absolute integral closure plays an extremely important role in perfectoid ring theory.

\begin{defi}
    Let $A$ be an integral domain with the field of fractions $Q(A)$. Let the algebraic closure of $Q(A)$ be denoted by $\overline{Q(A)}$. The integral closure of $A$ in $\overline{Q(A)}$ is denoted by $A^+$ and is called the absolute integral closure of $A$.
\end{defi}

It can be seen that the absolute integral closure of the valuation ring is an example of a Pr\"ufer domain.

\begin{prop}[\cite{Bourbaki}*{Chap.VI, \S 8.6, Prop.6}]\label{prop:A^+ is prufer}
    Let $A$ be a valuation ring with a valuation $v$ on a field $K$, $L$ an algebraic extension of $K$ and $A'$ the integral closure of $A$ in $L$. There is a one-to-one correspondence between the valuation rings of valuations on $L$ that extend $v$ and the localizations ${A'}_{\ideal{m}}$ at each maximal ideal $\ideal{m}$ of $A'$. In particular, $A'$ is Pr\"ufer.
\end{prop}

Also, by virtue of a valuation ring $A$ itself being a Pr\"ufer domain, it can be shown that $A \to A^+$ is faithfully flat. Although this is well known, let us provide a proof.

\begin{prop}\label{prop:A -> A^+ is f.f.}
    Let $A$ be a valuation ring and let $A^+$ be the absolute integral closure of $A$. The natural map $A\to A^+$ is faithflly flat.
\end{prop}

\begin{proof}
    The map $A \to A^+$ is torsion-free. This is because, for any $a \neq 0$ in $A$ and $\alpha \in A^+$, there exist elements $c_1, \dots, c_r \in A$ such that
    \[\alpha^r + c_1\alpha^{r-1}+\dots+c_r = 0\]
    Here, let $r$ be the smallest degree of such a monic polynomial. By multiplying this equation by $a$, it follows that $c_r = 0$, which gives $\alpha (\alpha^{r-1} + c_1 \alpha^{r-2} + \dots + c_{r-1}) = 0$. By the minimality of $r$, we conclude that $\alpha = 0$.

    Since $A$ is also a Pr\"ufer domain, $A^+$ is flat over $A$.

    Now, let $v$ be a valuation of $A$ on the field $K$. Then, for any maximal ideal $\ideal{m}'$ of $A^+$, by \ref{prop:A^+ is prufer}, there exists an extension $v'$ of $v$ to an algebraic extension $K'$ of $K$, along with a valuation ring $A'$ such that $A' = A^+_{\ideal{m}'}$. Then, for any maximal ideal $\ideal{m}$ of $A$, it must hold that for any $a \in \ideal{m}$ and $\alpha \in A^+$, we have $v'(a\alpha) = v(a)+v'(\alpha)> 0$. Consequently, $A^+ \neq \ideal{m} A^+$, and hence $A \to A^+$ is faithfully flat.
\end{proof}

By Kunz's theorem \cite{Kunz}, the importance of the flatness of the Frobenius map is widely recognized, and \cite{Lurie} provides an equivalent condition for the flatness of the Frobenius map $A/\varpi A \to A/\varpi^p A$ when $A$ is Noetherian and $\varpi$ is a regular element for which $\varpi^p$ divides $p$. Inspired by this discussion, \cite{shimomoto} examines the flatness of the Frobenius map in the case of valuation rings, leveraging \ref{prop:A -> A^+ is f.f.}.

\begin{prop}[\cite{shimomoto}*{Proposition 7.12.}]
    Let $A$ be a ring and let $p>0$ be a prime number. Assume that $\varpi\in A$ is an element such that $A/bA$ is a ring of characteristic $p>0$, and assume that there is a unit $u\in A^\times$ suth that $\varpi^p = bu$. Assume there is a faithfully flat $A$-algebra $B$ such that the Frobenius map $B/bB \to B/bB$ is surjective with kernel being equal to $(\varpi)$. Then the Frobenius map $A/\varpi A\to A/pA$ is flat. In paticular, if $A$ is a valuation ring, then the Frobenius map is flat. 
\end{prop}

In this context, the proof that the absolute integral closure of a valuation ring is a faithfully flat $A$-algebra (\ref{prop:A -> A^+ is f.f.}) relies on the fact that torsion-free modules over a valuation ring are flat. This leads to the proposal of the following problem.

\begin{prob}[\cite{shimomoto}*{Open Problem 7.13.}]
    Characterize a ring $A$ with the following property: If $M$ is any torsion-free $A$-module, then $M$ is flat.
\end{prob}

This problem is solved by \ref{thm:main} in this paper. This property is equivalent to $A$ being semi-hereditary.

The following problems are also related.

\begin{prob}[\cite{ISHIRO2022106898}*{Question 1.}]
    Let $A$ be a ring of mixed characteristic $p>0$. Assume that there is an element $\varpi \in A$ together with an element $u\in A^\times$ such that $\varpi^p = pu$. Characterise a ring $A$ such that the Frobenius map $A/\varpi A \to A/pA$ is faithfully flat, other than the case $A$ is a regular ring, a valuation ring, or a ring whose $p$-adic completion is an integral perfectoid ring.
\end{prob}

\begin{prob}[\cite{shimomoto}*{Open Problem 7.11.}]
    Let $A$ be a ring and let $p>0$ be a prime number. Assume that $\varpi\in A$ is an element such that $A/bA$ is a ring of characteristic $p>0$, and assume that there is a unit $u\in A^\times$ suth that $\varpi^p = bu$. Characterise a ring $A$ such that the Frobenius map $A/\varpi A \to A/bA$ is flat.
\end{prob}

\subsection*{Acknowledgments}
I would like to thank Hisanori Ohashi and Yuya Matsumoto for attending the seminar and giving me various suggestions. I would also like to express my deepest gratitude to my company, SKILLUP NeXt Ltd, and its members.

\end{document}